\documentclass[11pt]{amsart}

\usepackage{latexsym,mathrsfs,bm}
\usepackage[english]{babel}
\usepackage{paralist}
\usepackage{ulem} 
\usepackage{amsfonts,amssymb,amsmath}
\usepackage{amsthm}
\usepackage{graphicx}

\usepackage[latin1]{inputenc}
\usepackage{bm}
\usepackage{color}

\newcommand{\eps}{\varepsilon}
\newcommand{\R}{\mathbb{R}}
\newcommand{\Q}{\mathbb{Q}}
\newcommand{\C}{\mathbb{C}}

\newcommand{\N}{\mathbb{N}}

\newcommand{\Z}{\mathbb{Z}}

\newcommand{\cC}{\mathcal{C}}

\renewcommand{\mod}[1]{~\pr{\textnormal{mod}~#1}}

\newtheorem*{theo*}{Theorem}
\newtheorem{theo}{Theorem}

\newtheorem{ezer}{Exercise}

\newtheorem{prop}[ezer]{Proposition}

\newtheorem{lemma}{Lemma}
\newtheorem{corol}[lemma]{Corollary}

\theoremstyle{remark}
\newtheorem{remark}{Remark}

\newtheorem*{rem*}{Remark}

\newcommand{\pr}[1]{\left( #1\right)}

\newcommand{\pmd}[1]{\left| #1\right|}

\newcommand{\e}[1]{\operatorname{e}\pr{ #1}}

\newcommand{\df}{\mathrm{d}}

\newcommand{\comment}[1]{}



\newcommand{\syf}[1]{\left(\!\left(#1\right)\!\right)}
\newcommand{\floor}[1]{{\left\lfloor {#1} \right\rfloor}}
\newcommand{\abs}[1]{{\left| {#1} \right|}}

\let\originalleft\left
  \let\originalright\right
\renewcommand{\left}{\mathopen{}\mathclose\bgroup\originalleft}
  \renewcommand{\right}{\aftergroup\egroup\originalright}

\numberwithin{equation}{section}

\title{Partial sums of the cotangent function}

\author{S. Bettin}
\address{DIMA - Dipartimento di Matematica, Via Dodecaneso, 35, 16146 Genova - ITALY}
\email{bettin@dima.unige.it}

\author{S. Drappeau}
\address{SD: Aix Marseille Universit\'e, CNRS, Centrale Marseille, I2M UMR 7373, 13453 Marseille, France}
\email{sary-aurelien.drappeau@univ-amu.fr}

\begin{document}

\begin{abstract}
  Nous prouvons l'existence de formules de réciprocité pour des sommes de la forme~$\sum_{m=1}^{k-1} f(\frac{m}k) \cot(\pi\frac{mh}k)$, où~$f$ est une fonction~$C^1$ par morceaux, qui met en évidence un phénomène d'alternance qui n'apparaît pas dans le cas classique où~$f(x) = x$. Nous déduisons des majorations de ces sommes en termes du développement en fraction continue de~$h/k$.

  We prove the existence of reciprocity formulae for sums of the form $\sum_{m=1}^{k-1}f\pr{\frac{m}{k}}\cot\pr{\pi \frac{m h}k}$ where $f$ is a piecewise~$C^1$ function, featuring an alternating phenomenon not visible in the classical case where~$f(x)=x$. We deduce bounds for these sums in terms of the continued fraction expansion of~$h/k$.
\end{abstract}

\maketitle

\section{introduction}

There are several results in the literature proving reciprocity formulae for certain averages of the cotangent function. A prototypical example is the classical Dedekind sum which can be defined as
\begin{equation*}
  s\pr{\frac h k}:=-\frac1{4k}\sum_{m=1}^{k-1}\cot\pr{ \pi\frac{ m}k}\cot\pr{\pi \frac{m h}k},\qquad h, k\in\N,\ (h,k)=1,
\end{equation*}
and satisfies the well known reciprocity formula
\begin{equation*}
  s\pr{\frac h k}+s\pr{\frac{k}{h}}-\frac1{12hk}=\frac1{12}\pr{\frac hk+\frac kh-3}.
\end{equation*}
The Dedekind function has been generalized in several ways, all satisfying some sort of reciprocity (see e.g.~\cite{Zagier1973, Beck2003, Berndt1976}).

Another related example is given by the Vasyunin sum
\begin{equation*}
  V\pr{\frac h k}:=\sum_{m=1}^{k-1}\frac{m}k\cot\pr{\pi \frac{m\overline h}k},\qquad (h, k\in\N,\ (h,k)=1),
\end{equation*}
where, here and in the following, the overline indicates any multiplicative inverse modulo the denominator. The Vasyunin sum satisfies the reciprocity formula
\begin{equation}\label{eq:uvf}
  V\pr{ \frac hk}+V\pr{ \frac kh}=\frac{\log2\pi-\gamma}{\pi}\pr{k+h}+\frac{k-h}{\pi}\log\frac hk-\frac{\sqrt{hk}}{\pi^2}\int_{-\infty}^\infty\pmd{\zeta\pr{\tfrac12+it}}^2\pr{\frac hk}^{it}\frac{dt}{\frac14+t^2},
\end{equation}
where $\zeta(s)$ is the Riemann zeta-function, as well as another one relating $V\pr{ \overline h/k}$ with $V\pr{\overline k/h}$. For this, other generalizations, and the relation of $V$ with the B\'aez-Duarte and Nyman-Beurling criterion for the Riemann hypothesis we refer to~\cite{Vasyunin1995,Bettin2013a,BettinConrey2013,Baez-Duarte2003,Bagchi2006}.

Following the Euclid algorithm and repeatedly applying these results, one can then obtain bounds and asymptotic formulae for these sums in terms of the continued fraction expansion of $h/k$. See, for example~\cite{Hickerson1977} and~\cite{Bettin2015}.

All of the results mentioned above involve sums over all (non-zero) residues. Having in mind an application to the value distribution of Kashaev's knot invariants~\cite{BettinDrappeaua}, we are led to the problem of bounding \emph{partial} sums of the cotangent function. For example, let
\begin{equation*}
  C_\ell\pr{\frac h k}:=\frac1k\sum_{1\leq m\leq \ell}\cot\pr{\pi\frac{m  h}{k}},\qquad h\in\Z, k\in\N,\ (h,k)=1,\ 1\leq \ell<k.
\end{equation*}
Despite the simplicity of this definition, this partial average behave rather erratically, and there is little reason to suspect at first the existence of pure reciprocity formulae such as~\eqref{eq:uvf}.

\begin{figure}[h]
  \includegraphics[width=0.465\textwidth]{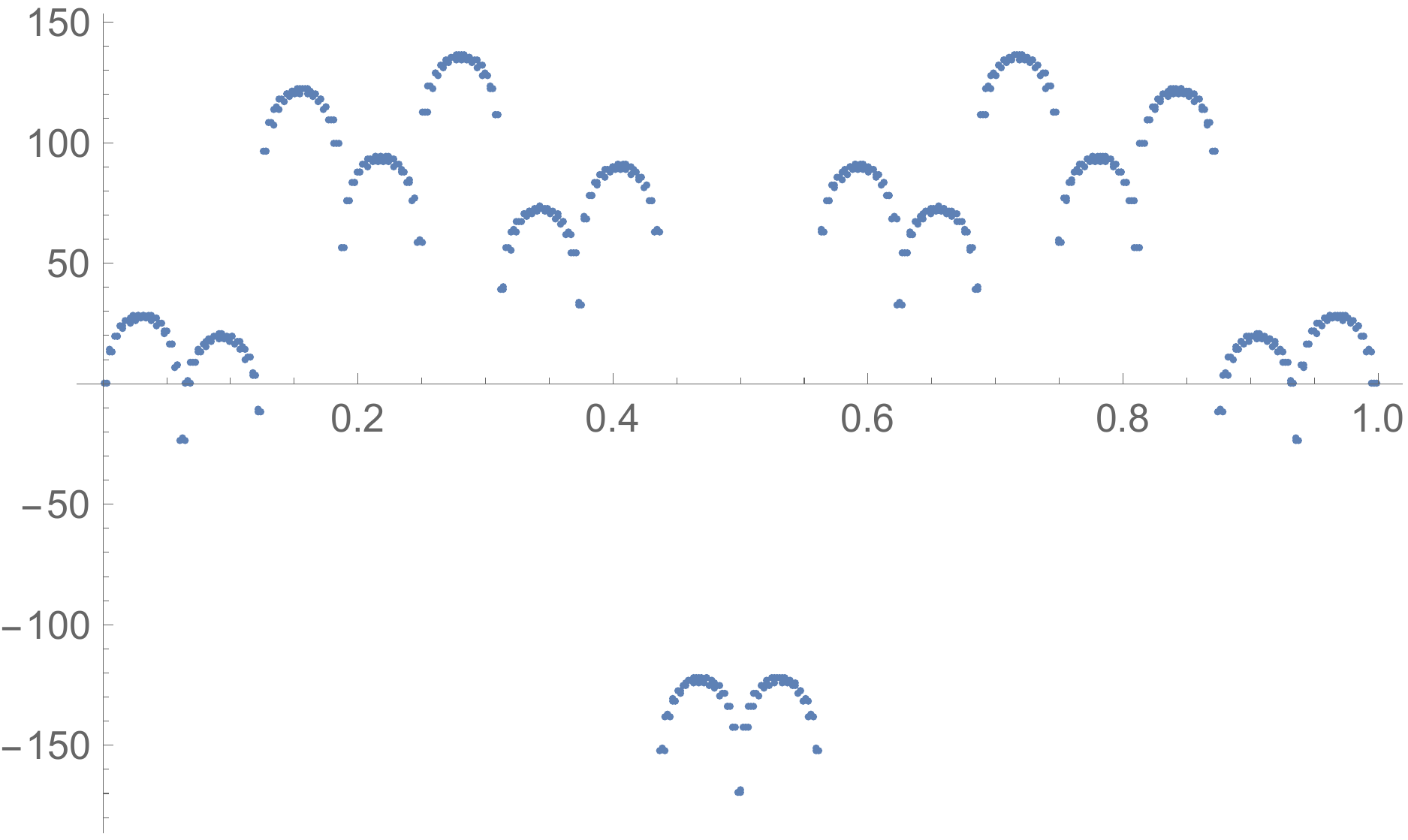}\hspace{0.05\textwidth}
  \includegraphics[width=0.465\textwidth]{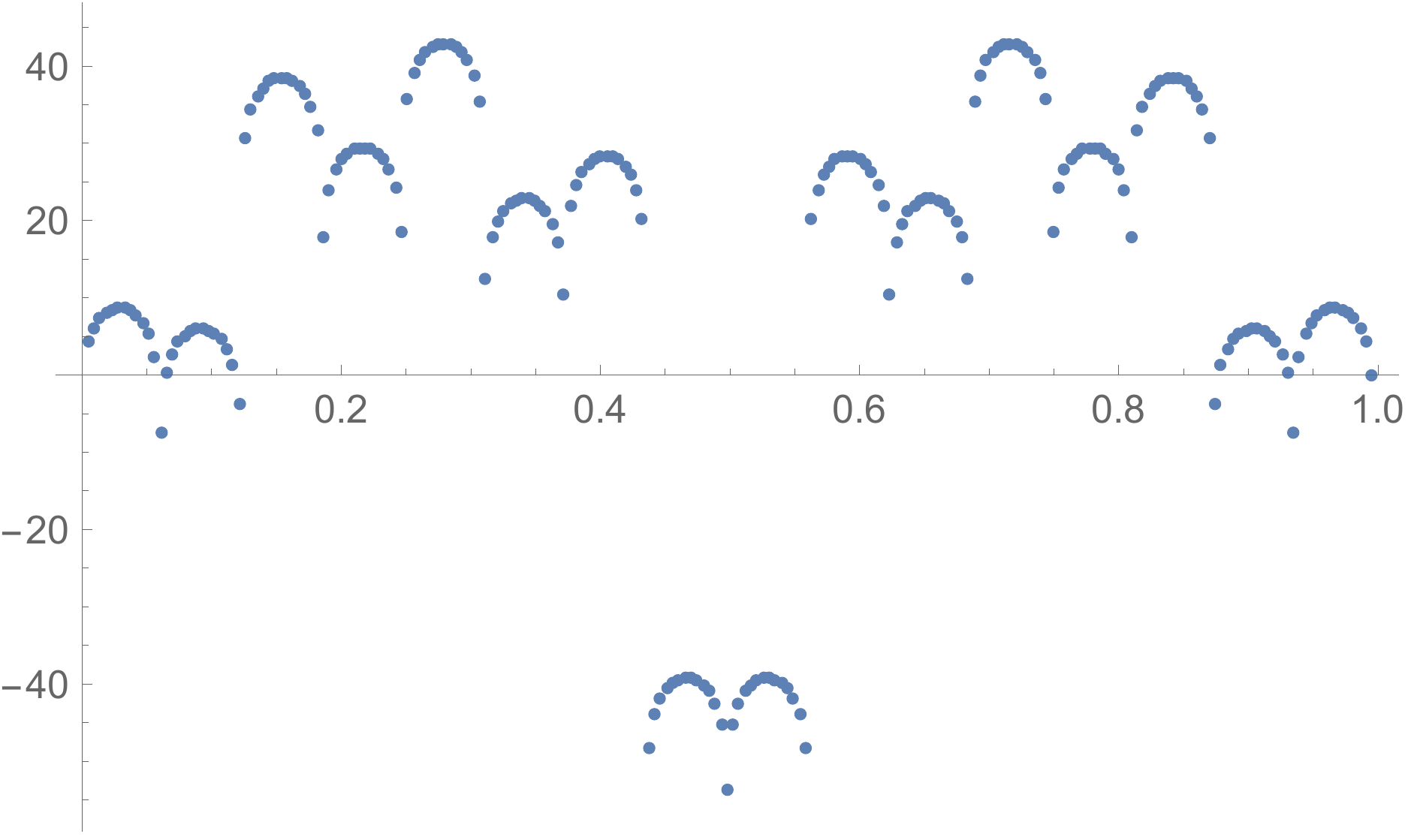}
  \caption{The graphs of $(\frac{\ell}{677},C_\ell(\frac{231}{677}))$ as $1\leq \ell<677$ and of $(\frac{\ell}{215},C_\ell(\frac{16}{215})$ as $1\leq \ell<215$.}
  \label{fig:1}
\end{figure}

Figure~\ref{fig:1} shows the evolution of $C_\ell({231}/{677})$ and $C_\ell({16}/{215})$ as $\ell$ varies. The continued fraction expansions of $\frac{231}{677}$ and $\frac{16}{215}$ are $\{0; 2, 1, 13, 2, 3, 2\}$ and $\{0;13, 2, 3, 2\}$ respectively, so that the similarity of the two graphs suggests that a reciprocity formula is in action. As we will see below, in this case, the reciprocity formula which we naturally obtain doesn't directly relate $C_\ell\pr{ h/ k}$ with $C_\ell\pr{{k}/{h}}$, nor $C_\ell\pr{{ h}/{k}}$ with $C_\ell\pr{{\overline k}/{h}}$, but rather $C_\ell\pr{{\overline h}/{k}}$ with $C_{\ell'}({\overline {h}}/k_1)$, where $k_1\equiv k_1\mod h$ and $0<k_1\leq h$, that is, the \emph{double} iteration of a standard reciprocity formula. This seems to be the first observed instance of such an alternating behaviour in these objects. We give this (non-exact) reciprocity formula in Corollary~\ref{cor:mcc} below. 

In general we will consider this question for 
\begin{equation*}
  S_f\pr{\frac h k}:=\frac1k\sum_{m=1}^{k-1}f\pr{\frac{m}{k}}\cot\pr{ \pi\frac{m h}k},\qquad h\in\Z, k\in\N,\ (h,k)=1.
\end{equation*}
where $f$ is a piecewise $\cC^1$ function. It will follow from our arguments that for generic functions $f$, $S_f\pr{\frac h k}$ doesn't satisfy a simple reciprocity formula unless $f$ is smooth with only one point of discontinuity other than, possibly, $0$. However, we are able to control its size in all cases, as needed for example in our application to the distribution of the Kashaev invariants.

\begin{theo}\label{th:mt}
  Let $f:\R\to\R$ be a $1$-periodic function which is piecewise~$\cC^1$, and continuous except possibly at $d$ points. Let $D_0=\max_{x\in\R}\lim_{\eps\to 0}|f(x+\eps)-f(x-\eps)|$ and $D_1:=\|f'\|_2$. Then
  \begin{align*}
    \abs{S_f\pr{\frac h k}}&\leq \frac{dD_0}{\pi k}\sum_{m=1}^rv_m\log ( \tfrac{v_m}{v_{m-1}})+O(dD_0+D_1),\\
    \abs{S_f\pr{\frac{\overline h}{k}}}&\leq \frac{dD_0}\pi\sum_{m=0}^{r-1}\frac{\log ( \frac{v_{m+1}}{v_{m}})}{v_m}+O(dD_0+D_1),
  \end{align*}
  where $v_0,\dots,v_m$ are the partial quotients of the continued fraction expansion of $h/k$.
\end{theo}

\begin{remark}
  Following the approach of~\cite{Bettin2015}, it would be possible to compute the distribution of $S_f\pr{\frac h k}$ as $1\leq h<k$ and $k\to\infty$ (the method of~\cite{Maier2016} might also be used with some effort). In particular, one can prove that there is a function $\eps:\R_{>0}\to\R_{>0}$ with~$\lim_{A\to\infty}\eps(A)=0$, so that for each~$A>0$, the number of~$h\in[1, k]$ with~$(h, k)=1$ and~$\abs{S_f\pr{\frac h k}}\leq A$ is at least~$k(1-\eps(A)) + o(k)$ as~$k\to\infty$.
\end{remark}

\begin{remark}
  If $\{x\}$ denotes the fractional part of $x$, then $\pi\cot(\pi x)-\{x\}^{-1}-\{-x\}^{-1}$ extends to an odd bounded function on~$\R$. In particular, Theorem~\ref{th:mt} holds also if one replaces $\cot(\pi y)$ by $ (\{y\}^{-1}-\{y\}^{-1})/\pi$ in the definition of $S_f\pr{\frac h k}$ (or any other $1$-periodic function of~$y$ having a similar asymptotic behaviour around~$0$).
\end{remark}

Theorem~\ref{th:mt} is a crucial ingredient in the following law of large numbers for values of the Kashaev invariants of the figure eight knot, which we prove in~\cite{BettinDrappeaua}~:
\begin{theo}
  For~$x\in\Q$, let
  $$ J(x) := \sum_{n=0}^\infty \prod_{r=1}^n \abs{1-{\rm e}^{2\pi i rx}}^2 $$
  be the Kashaev invariant of the~$4_1$ knot. For some constant~$\mu>0$, as~$Q\to +\infty$, we have
  $$ \log J(x) \sim \mu (\log Q) \log\log Q $$
  for a proportion~$1-o(1)$ of fractions~$x\in(0, 1]\cap\Q$ of reduced denominator at most~$Q$.
\end{theo}

\section{A reciprocity formula}

\subsection{Lemmas}

We introduce some notation and give some basic results on the Hurwitz and periodic zeta-functions (cf.~\cite[Chapter~12.9]{Apostol1976}). For $x\in\R$ and $\Re(s)>1$, let
\begin{equation*}
  \zeta(s,x):=\sum_{n+x>0}\frac1{(n+x)^s},\qquad
  F(s,x):=\sum_{n\geq1}\frac{\e{nx}}{n^s},
\end{equation*}
where $\e{x}:=e^{2\pi i x}$. Notice that we have ``periodized'' the Hurwitz zeta-function. It is well known that $\zeta(s,x)$ and $F(s,x)$ extend as meromorphic functions on $\C$ and satisfy functional equations. The functional equations are nicely expressed in terms of $\zeta^{\pm}(s,x):=\zeta(s,x)\pm\zeta(s,-x)$ and $F^{\pm}(s,x)=F(s,x)\pm F(s,-x)$. Indeed, they become 
\begin{align*}
  F^+(s,x)&=\chi(s)\zeta^+(1-s,x),\qquad
  F^-(s,x)=2i\frac{\Gamma(1-s)}{(2\pi)^{1-s}}\cos\pr{\frac{\pi s}2}\zeta^-(1-s,x).
\end{align*}
where $\chi(s):=2\frac{\Gamma(1-s)}{(2\pi)^{1-s}}\sin(\frac{\pi s}2)$ is as in the functional equation of the Riemann zeta-function.

We also recall the special values 
\begin{equation}\label{eq:sv1}
  \begin{aligned}
    F^{+}(1,x)& =-\log((1-e^{-2\pi i x})(1-e^{2\pi i   x}))=-\log(4\sin(\pi x)^2),\quad x\notin\Z,\\
    F^-(1, x)&=2\pi i\sum_{n\in\N}\frac{\sin(2\pi nx)}{\pi n} =2\pi i\syf{ x}
  \end{aligned}
\end{equation}
where $\syf{x}:=0$ if $x\in\Z$ and $\syf{x}:=\frac12-\{x\}$ otherwise.

Also, we have the expansion at $s=1$
\begin{equation*}
  \zeta(s,x)=\frac1{s-1}-\psi(\{x\})+O(s-1),
\end{equation*}
where $\psi$ is the digamma function~\cite[§8.36]{GZ}, and so also
\begin{equation}\label{eq:sva1}
  F^+(1-s,x)=-1-\big(\gamma+\log 2\pi +\tfrac12\psi(\{x\})+\tfrac12\psi(\{-x\})\big)(1-s)+O((s-1)^2).
\end{equation}

For $h,\ell\in\Z$, $k\in\N$ we define
\begin{equation}\label{eq:def-v1}
  V_1\pr{\frac hk,\frac\ell k}:=2\sum_{m\geq1}\sum_{n\geq1}\frac{\sin(2\pi n m\frac{h}k)\cos(2\pi m\frac{\ell}k)}{\pi nm}=2\sum_{m\geq1}\frac{\cos(2\pi m\frac{\ell}k)}{ m}\syf{\frac{mh}k}
\end{equation}
which clearly converges since $\sum_{m=1}^k\cos(2\pi m\ell/ k)\syf{ mh/k}=0$ by parity.
Actually, the above computation together with the functional equations for $\zeta(s,x)$ can be used to show that one can smoothly truncate the sum over~$m, n$ at height  $\ll Xk^{1+\eps}$, for any $X\geq1$, at the cost of an error which is $O((Xk)^{-100})$, as is done \emph{e.g.} in~\cite[p.~11423]{Bettin2015} with an analogous series.

For $h,p\in\Z$, $k,q\in\N$, $(h,k)=1$ we also define
\begin{align*}
  V_2(\tfrac hk,\tfrac pq)&:=\sum_{m\in\Z\atop |m+\frac pq|\geq 1}\sum_{n\geq 1}\frac{\sin(2\pi \frac hk n|m+\frac pq|)}{\pi n|m+\frac pq|}=\sum_{m\in\Z\atop |m+\frac pq|\geq 1}\frac{\syf{ \frac hk|m+\frac pq|}}{ |m+\frac pq|}
\end{align*}
where the outer sum is computed by summing together the terms $m$ and $-m$. Again, one easily sees that the series converges and that one can smoothly truncate the sums at $m,n\ll X(qk)^{1+\eps}$ at a cost of an error which is $O((Xqk)^{-100})$. 
The manipulation of conditionally convergent sums and integrals in the following is justified by these considerations.

We remark (cf. e.g. Lemma~\ref{lem:ml} below with $f(x)=x$ for $0< x<1$) that $V_1(\tfrac hk,\tfrac \ell k)$ and $V_2(\tfrac hk,\tfrac pq)$ reduce to $-\frac{\pi}k V(\tfrac hk)$ if $k|\ell$ and $p=0$.

We will prove a reciprocity formula relating $V_1$ with $V_2$, generalizing~\eqref{eq:uvf}, following the same approach as in the proof of~\cite[Theorem~5]{Bettin2013a}. We shall need the following uniform convexity bound for $F^+(\frac12+it,x)$. 

\begin{lemma}\label{lem:cb}
  Let $x\not\in\Z$ and
  \begin{equation*}
    K^+(s,x):=F^+(s,x)-\chi(s)(\{x\}^{s-1}+\{-x\}^{s-1}).
  \end{equation*}
  Then $K^+(\frac12+it,x)\ll_\eps (1+ |t|)^{\frac14+\eps}$ uniformly in $x$.
\end{lemma}
\begin{proof}
  For $\Re(s)=1+\eps$ we have $K^+(s,x)\ll_\eps 1$, whereas on $\Re(s)=-\eps$ after applying the functional equation and expanding the resulting series we have $K^+(s,x)\ll_\eps (1+|s|)^{\frac12+\eps}$ uniformly in $x$. Phragmen-Lindel\"of's theorem~\cite[§5.61]{Titchmarsh1939} then gives the claimed bound.
\end{proof}

\subsection{Main reciprocity formula}

\begin{prop}\label{prop:mp}
  Let $\ell,h,k\geq1$ and $(h,k)=1$. For $h\neq1$, let $\beta:=\{\frac kh\}^{-1}\{\frac \ell h \}$ if $k\nmid \ell$ and $\beta=0$ if $k\mid \ell$.  
  Then, we have
  \begin{equation}\label{eq:recipmain}
    \frac1hV_1\pr{\frac{h}k,\frac\ell k}+\delta_{h\neq1}\frac 1k V_2\pr{\Big\{\frac{k}h\Big\},\beta}=\Big(\frac{\gamma_{h,k,\ell}}h-\frac{1}{k}\Big)\log \pr{\frac{ k}{h}}+O\pr{\frac1k+\frac1h},
  \end{equation}
  where $\delta_{h\neq1}=0$ if $h=1$ and $\delta_{h\neq1}=1$ otherwise, and where $\gamma_{h,k,\ell}=1$ if $k\mid \ell$ and otherwise $\gamma_{h,k,\ell}=2$ if $k\leq h$ and  $0\leq\gamma_{h,k,\ell}\leq 1$ if $k> h$.
\end{prop}
\begin{remark}
  For $k\nmid\ell$, the term ${\gamma_{h,k,\ell}\log\pr{\frac{k}{h}}}/h$ can be replaced in this estimate by 
  $$\tfrac1h \big({\log\{\tfrac\ell k\}+\log\{-\tfrac\ell k\}-\log^- (\{\tfrac\ell k\}\tfrac kh)-\log^- (\{-\tfrac\ell k\}\tfrac kh)}\big),$$
  where $\log^-(x):=\min(\log x,0)$.
\end{remark}
\begin{proof}
  We assume $k\nmid \ell$, since otherwise the formula is an immediate consequence of the usual Vasyunin's formula~\eqref{eq:uvf}.
  Let
  \begin{equation*}
    I(h,k):=\frac1{2\pi i}\int_{(\frac12)}\frac{F^+(s,\frac{\ell}k)\zeta(1-s)}{h^{s}k^{1-s}}\,\frac{ds}{s(1-s)},
  \end{equation*}
  where $\int_{(c)}\cdot\, ds:=\int_{c-i\infty}^{c+i\infty} \cdot\, ds$.
  We will now proceed to evaluate $I(h,k)$ in two different ways.
  
  Throughout, we denote
    $$ \alpha = \frac{\ell}k. $$
  In the notation of Lemma~\ref{lem:cb}, we have
  \begin{equation*}
    I(h,k)=\frac1{2\pi i}\int_{(\frac12)}\frac{K^+(s,\alpha)\zeta(1-s)}{h^{s}k^{1-s}}\,\frac{ds}{s(1-s)}+\frac1{2\pi i}\int_{(\frac12)}\frac{\chi(s)(\{\alpha\}^{s-1}+\{-\alpha\}^{s-1})\zeta(1-s)}{h^{s}k^{1-s}}\,\frac{ds}{s(1-s)}.
  \end{equation*}
  By Lemma~\ref{lem:cb}, the first integral is bounded by 
  \begin{equation*}
    \ll \frac1{\sqrt {hk}}\int_{\R}\frac{|\zeta(\frac12+it )|\,dt}{(1+|t|)^{\frac32}}\ll \frac1{\sqrt {hk}},
  \end{equation*}
  whereas the contribution of $\{\alpha\}$ to the second integral is  
  \begin{equation*}
    \frac1{2\pi i}\int_{(\frac12)}\frac{\chi(s)\{\alpha\}^{s-1}\zeta(1-s)}{h^{s}k^{1-s}}\,\frac{ds}{s(1-s)}=\frac1{k\{\alpha\}}\frac1{2\pi i}\int_{(\frac12)}\zeta(s)\pr{\{\alpha\}\frac k h}^s\,\frac{ds}{s(1-s)}.
  \end{equation*}
  If $\{\alpha\}\frac k h< 1$ we move the line of integration to $\Re(s)=+\infty$ obtaining a contribution from the double pole at $s=1$
  of 
  \begin{equation*}
    \frac{\log (\{\alpha\}\frac k h)+\gamma-1}{h}.
  \end{equation*}
  If $\{\alpha\}\frac k h\geq1$ then we move the integral to $\Re(s)=-\frac14$ passing through a simple pole at $s=0$.  The contribution of the residue is $-\frac1{2k\{ \alpha\}}=O(1/h)$, whereas  the integral on the new line contributes $O(\frac1{k\{\alpha\}}\pr{\{\alpha\}\frac k h}^{-1/4})=O(1/h)$ since $\zeta(-1/4+it)\ll (1+|t|)^\frac 34$.
  Thus, in both cases we find that the contribution of $\{\alpha\}$ to the second integral is
  \begin{equation*}
    \frac{\log^- (\{\alpha\}k/h)+O(1)}h.
  \end{equation*}
  Repeating the same computation for $\{-\alpha\}$ we then obtain our first expression for $I(h,k)$:
  \begin{equation}\label{eq:firstf}
    I(h,k)=\frac{\log^- (\{\alpha\}k/h)+\log^- (\{-\alpha\}k/h)+O(1)}h+O\pr{\frac1{\sqrt{kh}}}.
  \end{equation}

  Now we compute $I(h,k)$ in a second way.
  We split the integral into
  \begin{align*}
    &\frac1{2\pi i}\int_{(\frac12)}\frac{F^+(s,\alpha)\zeta(1-s)}{h^{s}k^{1-s}}\,\frac{ds}{s(1-s)}\\
    &\qquad= \frac1{2\pi i}\int_{(\frac12)}\frac{F^+(s,\alpha)\zeta(1-s)}{h^{s}k^{1-s}}\,\frac{ds}{1-s} +\frac1{2\pi i}\int_{(\frac12)}\frac{F^+(1-s,\alpha)\zeta(s)}{h^{1-s}k^{s}}\,\frac{ds}{1-s} =I_1+I_2,
  \end{align*}
  say, where in the second integral we made the change of variables $s\to 1-s$. We now consider $I_1$. We move the line of integration to $\Re(s)=5/4$ passing through a simple pole at $s=1$. We obtain
  \begin{align*}
    I_1&=-\frac{F^{+}(1,\alpha)}{2h}+\frac1{2\pi i}\int_{(5/4)}\frac{F^+(s,\alpha)\zeta(1-s)}{h^{s}k^{1-s}}\,\frac{ds}{1-s}.
  \end{align*}
  By~\eqref{eq:sv1},
  \begin{align*}
    F^{+}(1,\alpha)&=-\log(4\sin(\pi \alpha)^2)=-2\log\{ \alpha\}-2\log\{- \alpha\}+O(1).
  \end{align*}
  Also,
  \begin{equation*}
    \frac1{2\pi i}\int_{(5/4)}\frac{F^+(s,\alpha)\zeta(1-s)}{h^{s}k^{1-s}}\,\frac{ds}{1-s}=
    \sum_{m\in\Z_{\neq0}}\sum_{n\geq1}\frac{\e{m\alpha}}k\frac1{2\pi i}\int_{(5/4)}\chi(1-s)(hn|m|/k)^{-s}\,\frac{ds}{1-s}.
  \end{equation*}
  By the Mellin formula 
  \begin{equation*}
    \frac1{2\pi i}\int_{(5/4)}\chi(1-s)u^{-s}\,\frac{ds}{1-s}=\frac{\sin(2\pi u)}{\pi u}
  \end{equation*}
  we find
  \begin{align*}
    &\frac1{2\pi i}\int_{(2)}\frac{F^+(s,\alpha)\zeta(1-s)}{h^{s}k^{1-s}}\,\frac{ds}{1-s}=
    \sum_{n\geq1,m\in\Z_{\neq0}}\frac{\sin(2\pi hn|m|/k)\e{m\alpha}}{\pi hn|m|}=\frac 1h V_1\pr{\frac hk,\alpha}.
  \end{align*}
  Thus,
  \begin{align*}
    I_1&=\frac{\log\{\alpha\}+\log\{-\alpha\} + O(1)}{h}+\frac1h V_1\pr{\frac hk,\alpha}.
  \end{align*}

  As for $I_2$ we move the line integration to $\Re(s)=5/4$ passing through a double pole at $s=1$.  By~\eqref{eq:sva1} the pole contributes a residue
  \begin{equation}\label{eq:rs2}
    \frac{\psi(\{\alpha\})+\psi(\{-\alpha\})  +2\log (2\pi k/h)}{2k}=\frac{-\{\alpha\}^{-1}-\{-\alpha\}^{-1}  +2\log (k/h)+O(1)}{2k}
  \end{equation}
  since $\psi(x) = \frac1x + O(1)$ for~$0<x<1$.
  The contribution of the integral to $I_2$ is
  \begin{align*}
    \frac1{2\pi i}\int_{(5/4)}\frac{F^+(1-s,\alpha)\zeta(s)}{h^{1-s}k^{s}}\,\frac{ds}{1-s}&=
    \sum_{m\in\Z}\sum_{n\geq 1}\frac1{2\pi i h}\int_{(5/4)}\chi(1-s)\pr{\frac{kn}h|m+\alpha|}^{-s}\,\frac{ds}{1-s}\\
    &=\sum_{m\in\Z}\sum_{n\geq 1}\frac{\sin(2\pi \frac khn|m+\alpha|)}{\pi kn|m+\alpha|} \\
    &=\frac1k \sum_{m\in\Z} \frac{\syf{\frac kh \abs{m+\alpha}}}{\abs{m+\alpha}},
  \end{align*}
where in the first line we have applied the functional equation to $F^+$ and expanded the Dirichlet series.

  Now, we observe that the series sums to zero if $h=1$ (since $k\abs{m+\alpha}\in\Z$ in this case) and the claimed result easily follows. Thus, assume $h\neq1$.
  We isolate the terms $m\in\{-\floor{\alpha}, -\floor{\alpha}-1,\floor{-\beta},\floor{-\beta}+1\}$ from the sum, where
  $$ \beta:=\Big\{\frac kh\Big\}^{-1}\Big\{\frac \ell h \Big\}. $$
  The terms~$m\in\{-\floor{\alpha}, -\floor{\alpha}-1\}$ contribute
  \begin{equation*}
    \frac{\syf{\tfrac kh\{\alpha\}}}{k\{\alpha\}}+\frac{\syf{\tfrac kh\{-\alpha\}}}{k\{-\alpha\}}
  \end{equation*}
  whereas the contribution of $m'\in\{\floor{-\beta},\floor{-\beta}+1\}$ (with  $m'\neq-\floor{\alpha},-\floor{\alpha}-1$) is bounded by
  \begin{equation*}
    \frac1k\frac{\syf{\tfrac kh|m'+\alpha|}}{|m'+\alpha|}\ll \frac1k.
  \end{equation*}
  Next, we replace $\alpha$ in the denominator by $\beta$. The error in doing  so is 
  \begin{align*}
    &\frac1{ k}\sum_{m\neq -\floor{\alpha},-\floor{\alpha}-1,\floor{-\beta},\floor{-\beta}+1}\syf{\tfrac kh|m+\alpha|}\pr{\frac1{|m+\{\alpha\}|}-\frac1{|m+\beta|}}\ll\frac1k.
  \end{align*}
  We can then include again the terms $-\floor{\alpha}$ and~$-\floor{\alpha}-1$ (when they are different from $\floor{-\beta},\floor{-\beta}+1$) at the cost of an error which is analogously seen to be $O(1/k)$, obtaining
  \begin{equation*}
    \frac1k \sum_{m\in\Z}\frac{\syf{\frac kh \abs{m+\alpha}}}{|m+\alpha|}=\frac1k \sum_{m\in\Z,\atop m\neq\floor{-\beta},\floor{-\beta}+1} \frac{\syf{\frac kh \abs{m+\alpha}}}{|m+\beta|}+\frac{\syf{\tfrac kh\{\alpha\}}}{k\{\alpha\}}+\frac{\syf{\tfrac kh\{-\alpha\}}}{k\{-\alpha\}}+O\pr{\frac1k}.
  \end{equation*}
  By periodicity of the sine we then have
  \begin{equation*}
    \frac1k\sum_{m\in\Z,\atop m\neq\floor{-\beta},\floor{-\beta}+1}\frac{\syf{\frac kh \abs{m+\alpha}}}{|m+\beta|}= \frac1k \sum_{m\in\Z,\atop m\neq\floor{-\beta},\floor{-\beta}+1} \frac{\syf{\frac kh \abs{m+\beta}}}{|m+\beta|} =\frac1k V_2\pr{\Big\{\frac kh\Big\},\beta}+O\pr{\frac 1k}.
  \end{equation*}
  Finally, we observe that 
  \begin{align*}
    \frac{\syf{\tfrac kh\{\alpha\}}}{k\{\alpha\}}+\frac{\syf{\tfrac kh\{-\alpha\}}}{k\{-\alpha\}}&=\frac{1}{2k\{\alpha\}}+\frac{1}{2k\{-\alpha\}}-\frac{\{\tfrac kh\{\alpha\}\}}{k\{\alpha\}}-\frac{\{\tfrac kh\{-\alpha\}\}}{k\{-\alpha\}} + O\pr{\frac1h} \\
    &=\frac{1}{2k\{\alpha\}}+\frac{1}{2k\{-\alpha\}}+O\pr{\frac1h}.
  \end{align*}
Indeed, the third term is $O(1/h)$ (and similarly for the fourth) since, if $0\leq \ell'\leq k$  and $0\leq \ell''\leq h$ are such that $\ell'=\ell\mod k$ and $\ell''\equiv \ell'\mod h$, then $\{\tfrac kh\{\alpha\}\}/{k\{\alpha\}}=\ell''/(\ell'h)\leq \frac1h$. 

  By the above computations we then have
  \begin{align*}
    \frac1{2\pi i}\int_{(5/4)}\frac{F^+(1-s,\alpha)\zeta(s)}{h^{1-s}k^{s}}\,\frac{ds}{1-s}&=\frac 1kV_2(\{\tfrac kh\},\beta)+\frac{1}{2k\{\alpha\}}+\frac{1}{2k\{-\alpha\}}+O\pr{\frac1k+\frac1h}
  \end{align*}
  and so, adding the contribution of the residue~\eqref{eq:rs2} we have
  \begin{equation*}
    I_2=V_2(\{\tfrac kh\},\beta)/k+{\log (k/h)}/{k}+O\pr{\frac1k+\frac1h}.
  \end{equation*}
  Thus,
  \begin{align*}
    I(h, k) =I_1+I_2&=\frac{\log\{\alpha\}+\log\{-\alpha\}}{h}+\frac1hV_1(\tfrac hk,\alpha)+\frac 1kV_2(\{\tfrac kh\},\beta)+\frac{\log (k/h)}{k}+O\pr{\frac1k+\frac1h}.
  \end{align*}
  Comparing this with~\eqref{eq:firstf} we obtain the claimed identity, since for $0<x<1$, $y>0$ we have
  \begin{equation*}
    \log x+\log(1-x)-\log^- (x y)-\log^- ((1-x)y)=-\gamma \log y+O(1),
  \end{equation*}
  where $\gamma=2$ if $0<y\leq 1$ and $0<\gamma\leq 1$ otherwise.
\end{proof}

\subsection{Bound for~$V_1$ in terms of the continued fraction expansion}

We now iterate the relation~\eqref{eq:recipmain} in order to obtain a bound for individual values of~$V_1$.

\begin{corol}\label{cor:mpc}
  Let $\ell,h,k\in\N$ with~$h<k$ and $(h,k)=1$. Then, uniformly for~$\alpha\in\frac1k\Z$, we have
  \begin{equation*}
    \begin{aligned}
      \abs{V_1\pr{\frac { h}k,\alpha}}&\leq\sum_{m=0}^{r-1}\frac{\log ( v_{m+1}/v_{m})}{v_m}+O(1),
      \qquad
      \abs{V_1\pr{\frac {\overline h}k,\alpha}}\leq \frac1k\sum_{m=1}^rv_m\log ( v_m/v_{m-1})+O(1),
    \end{aligned}
  \end{equation*}
  where $v_m$ is the $m$-th partial quotient of $\frac{ h}k$.
\end{corol}
\begin{proof}
  By Proposition~\ref{prop:mp}, recalling the notation~$\beta := \{\frac kh\}^{-1}\{\frac\ell h\}$ if~$k\nmid\ell$ and~$\beta := 0$ otherwise, we have
  \begin{align}
    \Big|kV_1\pr{\frac hk,\frac \ell k}\Big|&\leq \delta_{h\neq 1}\Big|h V_2\pr{\Big\{\frac kh\Big\},\beta}\Big|+k\log \pr{\frac{ k}{h}}+O(k)&& \text{if }h\leq k,\label{eq:a1}\\
    \Big|hV_2\pr{\frac kh,\beta}\Big|&\leq \Big|kV_1\pr{\frac hk,\frac \ell k}\Big|+h\log \pr{\frac{ h}{k}}+O(h)& &\text{if }h\geq k,\ h\neq1.\label{eq:a2}  
  \end{align}
  Now, let $h/k=[0;b_1,\dots,b_r]$ be the continued fraction expansion of $h/k$, with~$b_r \neq 1$ if~$r>1$. Also, let $h^*\in[1,k]$ be such that $h^*\equiv \overline {(-1)^{r+1} h}\mod k$ (in particular $h^*/k=[0,b_r,\dots,b_1]$). The Euclid algorithm on $h^*$ and $k$ can be written as (see~\cite{Khintchine1963})
  \begin{align*}
    &v_r=k,\qquad v_{r-1}= h^*,\\
    &v_{\ell+1}=b_{\ell+1}v_{\ell}+v_{\ell-1},\qquad \ell=0,\dots r.
  \end{align*}
  Then, alternating the use of~\eqref{eq:a1}-\eqref{eq:a2} and the reduction modulo the denominator in $V_1$, we obtain
  \begin{equation}\label{eq:preq}
    k\Big|V_1\pr{\frac {\overline h}k,\alpha}\Big|
    \leq \sum_{m=1}^r\big(v_m \log ( v_m/v_{m-1})+O(v_m)\big)= \sum_{m=1}^rv_m \log ( v_m/v_{m-1})+O(k),
  \end{equation}
  as desired, where the last step follows since $v_{n-2}\leq v_{n}/2$ for all $n$. 

  Indicating with $u_m/v_m$ and $u_m'/v_m'$ the $m$-th convergents of $h/k$ and $h^*/k$ respectively (with $v_{-1}=v'_{-1}=0$), one has $k=v_sv'_{r-s}+v_{s-1}v'_{r-s-1}$ for all $0\leq s\leq r$ (see~\cite[p. 91-92]{Heilbronn1969}). In particular, $k/2\leq v_sv'_{r-s}\leq k$. Thus, by~\eqref{eq:preq} we have
  \begin{equation*}
    k\Big|V_1\pr{\frac {h}k,\alpha}\Big|
    \leq\sum_{m=1}^rv'_m \log ( v'_m/v'_{m-1})+O(k) \leq k \sum_{m=0}^{r-1}\frac{\log ( v_{m+1}/v_{m})}{v_m} +O(k).\qedhere
  \end{equation*}
\end{proof}

\section{Proof of Theorem~\ref{th:mt}}

We will deduce Theorem~\ref{th:mt} from Corollary~\ref{cor:mpc} and the following lemma, which shows that $S_f\pr{\frac h k}$ is, up to a small error, a linear combination of $V_1(\frac {\overline h}k,\frac{\ell_j}k)$ for various values of~$\ell_j$.

\begin{lemma}\label{lem:ml}
  Let $h\in\Z, k\in\N$ with $(h,k)=1$. Let $f:\R\to\R$ be a $1$-periodic function which is piecewise~$\cC^1$, and continuous everywhere except possibly at $d$ points $\frac{\ell_1}k,\dots,\frac{\ell_d}k$,  with $0\leq \ell_1<\cdots< \ell_d<k$, and assume $f$ is left and right differentiable at such points.
  Also, assume $2f(\frac{{\ell_i}}{k})=f(\frac{{\ell_i}}{k}+)+f(\frac{{\ell_i}}{k}-),$
  where $f(x\pm):=\lim_{\eps\to0^+}f(x\pm\eps)$.
  Then,
  \begin{align*}
    S_f\pr{\frac h k}
    &= \frac1\pi\sum_{j=1}^d V_1\pr{\frac{\overline h}k,\frac {\ell_i} k}
    \Big(f\Big(\frac{\ell_j}k+\Big)-f\Big(\frac{\ell_j}k-\Big)\Big) 
    +O\Big(\|f'\|_2^{1/2}\Big).
  \end{align*}
\end{lemma}
\begin{proof}
  We have the Fourier expansion
  \begin{equation*}
    f(x)=\sum_{n\in\Z} \hat f(n)\e{-nx}, \qquad x\in\R
  \end{equation*}
  where, for $n\neq0$,
  \begin{align*}
    \hat f(n) &= \sum_{j=1}^d\frac{f(\tfrac{\ell_j}k-)-f(\tfrac{\ell_j}k+)}{2\pi i n}\e{n{\ell_{j}}/k}  - \frac{\hat{f'}(n)}{2\pi i n},
  \end{align*}
  and~$\widehat{f'}(n) := \int_0^1 f'(y)\e{ny}\df y$. Note that, by the Bessel inequality,
  \begin{equation}
    \sum_{n\neq 0} \abs{\frac{\widehat{f'}(n)}{n}} \ll \|f'\|_2.\label{eq:bound-bessel}
  \end{equation}
  Now, we have
  \begin{equation*}
    \sum_{m=1}^{k-1} \cot\pr{\pi \frac{m h}k}\e{-\frac{nm}k}=-2ik\syf{\frac{n\overline h}k}
  \end{equation*}
  and thus,
  \begin{align*}
    \sum_{m=1}^{k-1}f\pr{\frac{m}{k}}&\cot\pr{\frac {\pi m h}k} \\
    & = -2i k \sum_{n\in\Z_{\neq0}} \syf{\frac{n\overline h}k} \bigg(
    \sum_{j=1}^d\frac{f(\tfrac{\ell_j}k-)-f(\tfrac{\ell_j}k+)}{2\pi i n}\e{n{\ell_{j}}/k} - \frac{\widehat{f'}(n)}{2\pi i n}\bigg).
  \end{align*}
  Grouping this with the definition~\eqref{eq:def-v1} and the bound~\eqref{eq:bound-bessel} yields our claim.
\end{proof}
\begin{proof}[Proof of Theorem~\ref{th:mt}]
  At the cost of committing an error of size $O(d D_0)$ in $S_f\pr{\frac h k}$, we modify $f$ so that all of its $d$ points of discontinuity are at rationals of the form $\frac{\ell}{k}$ with $2f\pr{\frac{{\ell}}{k}}=f\pr{\frac{{\ell}}{k}+}+f\pr{\frac{{\ell}}{k}-}$. Theorem~\ref{th:mt} then follows by Proposition~\ref{th:mt} and Lemma~\ref{lem:ml}.
\end{proof}

\subsection*{Particular case: partial cotangent sums}

For the partial cotangent sums, alluded to in the introduction (Figure~\ref{fig:1}), since there is only one point of discontinuity other than $0$, we indeed obtain a (non-exact) reciprocity relation in the following form.

\begin{corol}\label{cor:mcc}
  Let $h\in\Z, k\in\N$ with $(h,k)=1$, $1\leq h< k$, and let $0< \ell <k$. Let
  $k_1\equiv k\mod h$, $\ell'\equiv \ell \mod h$ with $1\leq k_1,\ell'\leq h$ and let $\ell_1\equiv \ell'\mod{k_1}$, $h_1\equiv h\mod {k_1}$ with $0\leq h_1< k_1$, $1\leq \ell_1\leq k_1$. Then,
  \begin{equation*}
    C_{\ell}(\overline{h}/k)-C_{\ell_1}(\overline{h_1}/{k_1})=\frac1\pi\Big({\gamma_{h,k,\ell}}-1\Big)\log \frac{ k}{h}+O\pr{\frac{h}{{k_1}}}.
  \end{equation*}
  with $0\leq \gamma_{h,k,\ell}\leq1$.
\end{corol} 
\begin{proof}
  By Lemma~\ref{lem:ml} we have $C_\ell(\overline h/k)=\frac1\pi (V_1(\frac{ h}{k},\frac{\ell }k)- V(\frac{ h}{k},0))+O(1)$. The result then follows, since applying Proposition~\ref{prop:mp} to $(\tfrac{h}k,\tfrac\ell k)$ and to $(\tfrac{h}{k_1},\tfrac{\ell'} {k_1})\equiv(\tfrac{h_1}{k_1},\tfrac{\ell_1} {k_1})\mod 1$ we obtain
  \begin{equation*}
    V_1\pr{\frac{h}k,\frac\ell k}-V_1\pr{\frac{h}{k_1},\frac{\ell'} {k_1}}
    =\Big({\gamma_{h,k,\ell}}-\frac{h}{k}\Big)\log {\frac{ k}{h}}+\frac{h}{k_1}\log {\frac{ k_1}{h}}+O\pr{\frac{h}{{k_1}}}.\qedhere
  \end{equation*}
\end{proof}

\bibliographystyle{amsalpha2}
\bibliography{../bib}

\end{document}